\documentclass[preprint,12pt,numbers,sort&compress]{elsarticle}

\usepackage{amssymb}
\usepackage{amsmath}
\usepackage{ntheorem}
\usepackage{graphicx}
\usepackage{booktabs}
\usepackage{multirow}
\journal{Applied Mathematics and Computation}
\newtheorem{theorem}{Theorem}
\newtheorem{remark}{Remark}
\newtheorem{example}{Example}
\newtheorem{lemma}{Lemma}
\newtheorem*{proof}{Proof}
\newtheorem{definition}{Definition}
\begin{document}
	
	\begin{frontmatter}
		
		\title{Stability analysis of Runge-Kutta methods for nonlinear delay-integro-differential-algebraic equations}
		\author[a]{Gehao Wang}	
		\ead{xtuwanggh@smail.xtu.edu.cn}
		\author[a]{Yuexin Yu \corref{cor1}}
		\ead{yuyx@xtu.edu.cn}	
		\affiliation[a]{organization={Hunan Key Laboratory for Computation and Simulation in Science and Engineering, School of Mathematics and Computational Science},
		addressline={Xiangtan University},
		city={Xiangtan},
		postcode={411105},
		state={Hunan},
		country={PR China}}
		\cortext[cor1]{Corresponding author}
		\begin{abstract}
			This paper is devoted to examining the stability of Runge-Kutta methods for solving nonlinear delay-integro-differential-algebraic equations  (DIDAEs). The stability of exact solution for nonlinear DIDAEs is obtained by using the Halanay's inequality. Hybrid numerical schemes combining Runge-Kutta methods and compound quadrature rules are analyzed for nonlinear DIDAEs. Criteria for ensuring the global and asymptotic stability of the proposed schemes are established. Several numerical examples are provided to validate the theoretical findings. 
		\end{abstract}

		
	\end{frontmatter}
		
		
		\section{Introduction}
		In scientific and engineering computations, mathematical models of many real-world problems involve not only delay effects, but also integral operators and algebraic constraints. These equations are collectively known as DIDAEs. DIDAEs are widely applied across multiple practical domains, including biomathematics, control theory, electric power systems, fluid dynamics, and constrained mechanical systems \cite{Niculescu2003,Brunner2004,Kolmanovskii2013}. For instance, in power system simulations, network topology and electromagnetic dynamics are often formulated as DIDAEs. Similarly, in biomathematics, population dynamics models incorporating memory effects can also be expressed in this form.
		
		Runge-Kutta methods, with their core advantages of high accuracy, strong stability, adaptability, and broad applicability, are among the preferred numerical tools for solving differential equations in science and engineering \cite{1,2}. The fundamental properties of DIDAEs are derived from both delay-integro-differential equations (DIDEs) and delay-differential-algebraic equations (DDAEs), and their research development relies on the established theoretical foundations of these two classes of equations. 
		
		Researchers have explored various approaches to analyzing DIDEs. Nonlinear stability analysis of neutral DIDEs was conducted by Wang using one-leg methods \cite{Wang2014} and Runge-Kutta methods \cite{Wang2011}, while Zhao employed block boundary value methods \cite{Zhao2013} and symmetric Runge-Kutta methods  \cite{Zhao2017} to study their stability properties. Yu \cite{Yu2025} further established stability criteria for neutral DIDEs discretized using general linear methods.  Kucche \cite{Kucche2019} investigated Ulam-type stability for nonlinear DIDEs, and Amirali \cite{Amirali2023} proposed a novel approach for establishing stability inequalities for high-order DIDEs. Nowak \cite{Nowak2023} analyzed the asymptotic stability of nonlinear neutral DIDEs, while Mahmoud \cite{Mahmoud2024} employed Lyapunov functionals to obtain asymptotic stability results for third-order nonlinear neutral DIDEs. Additionally, numerical techniques such as Runge-Kutta methods \cite{Yuan2013}, general linear methods \cite{HU2011} and block boundary value methods \cite{Zhao2014} have found extensive application in the stability analysis of DIDEs. 
		
		Research on DDAEs has also seen significant progress in both stability and asymptotic stability analyses. Earlier, Tian \cite{Tian2011} conducted early work on the asymptotic stability of general linear methods for DDAEs. Further contributions by Li \cite{LI2011} and Tian \cite{Tian2014}  examined the stability of Runge-Kutta methods for neutral DDAEs. Zhang studied the asymptotic stability \cite{Zhang2010BBV} and stability \cite{Zhang2021} of DDAEs using the block boundary value methods. Thuan \cite{Thuan2024} investigated the stability of stochastic DDAEs using Lyapunov functions and comparison principle, while Sawoor \cite{Sawoor2020} examined Lyapunov-based stability of linear DDAEs. Other numerical schemes, such as general linear methods \cite{Yu2015} and implicit Euler method \cite{Sun2014}, have also demonstrated unique advantages in analyzing the stability of DDAEs. 
		
		Compared with DDAEs and DIDEs, the stability analysis of DIDAEs is significantly more complex. Yuan \cite{Yuan2014} performed a stability analysis of two-step Runge-Kutta methods for linear neutral DIDAEs. Subsequently, Liu and Li \cite{liu2022} extended this work to a more general class of functional differential-algebraic equations (FDAEs), and systematically explored the asymptotic stability properties of Runge-Kutta methods. 
		
		Despite these advances, few results have been reported on the stability of numerical methods for nonlinear DIDAEs. To the best of our knowledge, only Yan \cite{Yan2020} has investigated the  stability of block boundary value methods for nonlinear non-stiff DIDAEs. However, stiffness, combined with delays and algebraic constraints, significantly complicates the development of reliable numerical methods. Without rigorous stability analysis, numerical solutions may exhibit spurious oscillations, instability, or divergence—thereby undermining their practical utility. Therefore, there is a strong need to develop and analyze numerical methods that are both stable and capable of addressing the complex structure of nonlinear DIDAEs. This paper aims to address these challenges by investigating the stability properties of both exact and numerical solutions for such systems.
		
		The remainder of this paper is organized as follows: Section 2 examines the stability and asymptotic stability of system using Halanay’s inequality. In Section 3, we investigate Runge-Kutta methods with compound quadrature rules, which provide a novel framework for DIDAEs analysis. Section 4 introduces several stability concepts and supporting lemmas essential for establishing the theoretical results. Section 5 presents the main results, elaborating on the criteria for Runge-Kutta methods with compound quadrature rules to achieve global and asymptotic stability.	Finally, Section 6 provides illustrative examples to demonstrate the practical applicability of the proposed methods.
		\section{DIDAEs and stability properties of the exact solution}
		This part introduces DIDAEs and the essential features of global and asymptotic stability behavior exhibited by exact solution.
		
		The symbols $\langle\cdot,\cdot\rangle$ and $\|\cdot\|$ represent a specified inner product and its associated norm in the complex space $\mathbb{C}^{N}$. It is notable that $N$ may be any positive integer.
		
		Consider the subsequent system of nonlinear DIDAEs with a constant delay $\tau > 0$ :
		\begin{equation}\label{2.1}
			\begin{cases}
				&u'(t)=f(t,\,u(t),\,\int_{t-\tau}^{t}K_{1}\left(t,\,\theta,\,u(\theta),\,v(\theta)\right)d\theta),\ t_{0}\leq t,\\&v(t)=g(t,\,u(t),\,\int_{t-\tau}^{t}K_{2}(t,\,\theta,\,u(\theta),\,v(\theta))d\theta),\  t_{0}\leq t,\\&u(t)=\psi(t), \ v(t)=\varphi(t),\  t_{0}-\tau\leq t\leq t_{0},\end{cases}
		\end{equation}
		where   $f:[t_0,+\infty)\times\mathbb{C}^{N_1}\times\mathbb{C}^{N_3}\to\mathbb{C}^{N_1},\ g:[t_0,+\infty)\times\mathbb{C}^{N_1}\times\mathbb{C}^{N_4}\to\ \mathbb{C}^{N_2},\  K_1:[t_0,+\infty)\times[t_0-\tau,+\infty)\times\mathbb{C}^{N_1}\times\mathbb{C}^{N_2}\to\mathbb{C}^{N_3}\  \text{and }\ K_2:[t_0,+\infty)\times[t_0-\tau,+\infty)\times\mathbb{C}^{N_1}\times\mathbb{C}^{N_2}\to\mathbb{C}^{N_4}$ are defined as functions with adequate smoothness, and $N_i\,(i=1,2,3,4)$ are positive integers.

		In order to disscuss the stability of DIDAEs (\ref{2.1}), we introduce another system with different initial condition:
		\begin{equation}\label{2.8}
			\begin{cases}&\tilde{u}'(t)=f(t,\,\tilde{u}(t),\,\int_{t-\tau}^{t}K_{1}(t,\,\theta,\,\tilde{u}(\theta),\tilde{v}(\theta))d\theta),\ t_{0}\leq t,\\&\tilde{v}(t)=g(t,\,\tilde{u}(t),\,\int_{t-\tau}^{t}K_{2}(t,\,\theta,\,\tilde{u}(\theta),\,\tilde{v}(\theta))d\theta),\ t_{0}\leq t,\\&\tilde{u}(t)=\tilde\psi(t),\  \tilde{v}(t)=\tilde\varphi(t),\  t_{0}-\tau\leq t\leq t_{0}.\end{cases}
		\end{equation}
		
		We hypothesize that the equations (\ref{2.1}) and (\ref{2.8}) fulfill the subsequent Lipschitz conditions with respective constants $\alpha$ and 
		$L_i>0,\ 1\leq i\leq 7\ $for all $t\in [t_0,+\infty),\ \theta\in [t_0-\tau,+\infty),\ u_1,\,\hat{u}_1,\,\tilde{u}_1,\,y_1,\, y_2\in\mathbb{C}^{N_1},\ \hat{v}_1,\,\tilde{v}_1\in\mathbb{C}^{N_2},\ \hat{p}_1,\,\tilde{p}_1,\,z\in\mathbb{C}^{N_3},\ \hat{q}_1,\, \tilde{q}_1\in\mathbb{C}^{N_4}$
		\begin{align}
			&\label{2.2}\|f(t,\,u_{1},\,\hat{p}_{1})-f(t,\,u_{1},\,\tilde{p}_{1})\|\leq L_{1}\left\|\hat{p}_{1}-\tilde{p}_{1}\right\|,\\
			\label{2.3}&\|g(t,\,\hat{u}_{1},\,\hat{q}_{1})-g(t,\,\tilde{u}_{1},\,\tilde{q}_{1})\|\leq L_{2}\left\|\hat{u}_{1}-\tilde{u}_{1}\right\|+L_{3}\left\|\hat{q}_{1}-\tilde{q}_{1}\right\|,\\
			\label{2.4}&\|K_{1}(t,\,\theta,\,\hat{u}_{1},\,\hat{v}_{1})-K_{1}(t,\,\theta,\,\tilde{u}_{1},\,\tilde{v}_{1})\|\leq L_{4}\left\|\hat{u}_{1}-\tilde{u}_{1}\right\|+L_{5}\left\|\hat{v}-\tilde{v}_{1}\right\|,\\
			\label{2.5}&\|K_{2}(t,\,\theta,\,\hat{u}_{1},\,\hat{v}_{1})-K_{2}(t,\,\theta,\,\tilde{u}_{1},\,\tilde{v}_{1})\|\leq L_{6}\left\|\hat{u}_{1}-\tilde{u}_{1}\right\|+L_{7}\left\|\hat{v}_{1}-\tilde{v}_{1}\right\|,\\
			\label{2.6}
			&\Re\langle y_1-y_2,\,f(t,\,y_1,\,z)-f(t,\,y_2,\,z)\rangle\leqslant\alpha\|y_1-y_2\|^2,
		\end{align}
		in which $(-\alpha) $ is given and nonnegative.
		
		 Moreover, the initial  functions for the problem (\ref{2.1}) $$\psi:[t_0-\tau,t_0]\to\mathbb{C}^{N_1},\ \varphi:[t_0-\tau,t_0]\to\mathbb{C}^{N_2},$$ and initial function for the perturbation problem (\ref{2.8}) $$\tilde{\psi}:[t_0-\tau,t_0]\to\mathbb{C}^{N_1},\ \tilde{\varphi}:[t_0-\tau,t_0]\to\mathbb{C}^{N_2}$$ are sufficiently smooth and meet the required consistency conditions
		\begin{equation}
			\begin{cases}\label{2.7}
				&\psi(t_0)=f(t_0,\,\psi(t_0),\,\int_{t_0-\tau}^{t_0}K_1(t_0,\,\theta,\,\psi(\theta),\,\varphi(\theta))d\theta),\\
				&\varphi(t_0)=g(t_0,\,\psi(t_0),\,\int_{t_0-\tau}^{t_0}K_2(t_0,\,\theta,\,\psi(\theta),\,\varphi(\theta))d\theta),\\
			\end{cases}
		\end{equation}
		and 
		\begin{equation}
			\begin{cases}
				&\tilde{\psi}(t_0)=f(t_0,\,\tilde{\psi}(t_0),\,\int_{t_0-\tau}^{t_0}K_1(t_0,\,\theta,\,\tilde{\psi}(\theta),\,\tilde{\varphi}(\theta))d\theta),\\
				&\tilde{\varphi}(t_0)=g(t_0,\,\tilde{\psi}(t_0),\,\int_{t_0-\tau}^{t_0}K_2(t_0,\,\theta,\,\tilde{\psi}(\theta),\,\tilde{\varphi}(\theta))d\theta).\\
			\end{cases}
		\end{equation}
		\begin{remark}
			In this setting, $\alpha $ functions as the one-sided  Lipschitz constant, while each $L_i\, (1\leq i\leq 7)$ acts as the classical Lipschitz constants. A prevailing assumption is that $L_i$ does not attain notably large positive values. Importantly, we permit massive number for the classical Lipschitz constant of $f(t,\,y,\,z)$ with respect to $y$; that is, the problems' stiffness is allowed to exist.
		\end{remark} 
		
		Before explaining our main results, we suppose that the problems (\ref{2.1}) and (\ref{2.8}) possess unique exact solutions, denoted by $\{u(t),\,v(t)\}$ and $\{\tilde{u}(t),\,\tilde{v}(t)\}$, respectively, and  we need the following  generalized Halanay's inequality.
		\begin{lemma}[\cite{Wanghalanay}]\label{hlemma3.1}
		 Consider inequalities
		\begin{align}
			y'(t) &\leq -A y(t) + B \max_{\xi \in [t-\tau,t]} y(\xi) + C \max_{\xi \in [t-\tau,t]} z(\xi),\quad t \geq t_0, \label{lemma2.9}\\
			z(t) &\leq G \max_{\xi \in [t-\tau,t]} y(\xi) + H \max_{\xi \in [t-\tau,t]} z(\xi),\quad t \geq t_0, \label{lemma2.10}
		\end{align}
		where \( t_0 \) is a constant. If \(A>0,\) \(B,\,C,\,G,\,H\geq 0 \) and \( H < 1 \), then for every \( \epsilon>0 \), there exist \( \delta(\epsilon)\to \delta_+ < 0, \epsilon \to 0_+ \), such that
		\begin{align}
			y(t) &\leq (1+\epsilon) \max_{\xi\in[t_0-\tau,t_0]} y(\xi)e^{\delta(\epsilon)(t - t_0)}, \quad t \geq t_0, \label{lemma2.11}\\
			z(t) &\leq (1+\epsilon) \max_{\xi\in[t_0-\tau,t_0]} z(\xi)e^{\delta(\epsilon)(t - t_0)}, \quad t \geq t_0 \label{lemma2.12}
		\end{align}
		for every nonnegative solution \( (y, z):[t_0-\tau, +\infty)\to \mathbb{R}^2_+ \) of the inequality \eqref{lemma2.9}--\eqref{lemma2.10} if and only if 
		\[
		-A + B + \frac{CG}{1 - H} < 0. \label{lemma2.13}
		\]
		\end{lemma}
		\begin{theorem}
			Suppose problems \eqref{2.1} and \eqref{2.8} satisfy conditions \eqref{2.2}--\eqref{2.6} with
			\begin{equation}
\alpha+L_1L_4\tau+\frac{L_1L_5\tau(L_2+L_3L_6\tau)}{1-L_3L_7\tau}<0,\quad L_3L_7\tau<1.
			\end{equation}
			Therefore, we obtain
			\begin{equation}\label{15}
				\begin{aligned}
					\|u(t) - \tilde{u}(t)\| &\leq  \mathcal{H}_1\max_{s \in [t_0 - \tau, t_0]} \|\psi(s) - \tilde{\psi}(s)\|, \\
					\|v(t) - \tilde{v}(t)\| &\leq \mathcal{H}_2\max_{s \in [t_0 - \tau, t_0]} \|\varphi(s) - \tilde{\varphi}(s)\|,
				\end{aligned}
			\end{equation}
			where $\mathcal{H}_i\,(i = 1, 2)$ are constants, and 
			\begin{equation}\label{16}
				\lim\limits_{t\rightarrow+\infty}\|u(t)-\tilde{u}(t)\|=0,\quad\lim\limits_{t\rightarrow+\infty}\|v(t)-\tilde{v}(t)\|=0.
			\end{equation}
		\end{theorem}
		\begin{proof}
		Define $U(t) = \|u(t) - \tilde{u}(t)\|$ and $V(t) = \|v(t) - \tilde{v}(t)\|$ for brevity.
		By conditions \eqref{2.2}--\eqref{2.6}, it is found that
		\begin{equation}\label{h3.4}
			U'(t)\leq \alpha U(t)+L_1L_4\tau\max\limits_{s \in [t-\tau,t]}U(s)+L_1L_5\tau\max\limits_{s \in [t-\tau,t]}V(s),
		\end{equation}
		and
		\begin{equation}\label{h3.5}
			\begin{aligned}
		V(t)\leq &L_2U(t)+L_3L_6\tau\max\limits_{s \in [t-\tau,t]}U(s)+L_3L_7\tau\max\limits_{s \in [t-\tau,t]}V(s)\\
		\leq&(L_2+L_3L_6\tau)\max\limits_{s \in [t-\tau,t]}U(s)+L_3L_7\tau\max\limits_{s \in [t-\tau,t]}V(s).
			\end{aligned}
		\end{equation}
		Based on Lemma \ref{hlemma3.1}, to prove the theorem, it is enough to derive \eqref{15} and \eqref{16} from \eqref{h3.4} and \eqref{h3.5}.
		\end{proof}

		\section{Runge–Kutta discretization}
		
		Regarding the nonlinear DIDAEs (\ref{2.1}), we initially revisit the $s$-stage fundamental Runge–Kutta method	
		\begin{equation}
			\begin{cases}\label{2.9}
				&y_i^{(n)}=y_n+h\sum\limits_{j=1}^sa_{ij}\,f(t_n+c_jh,\,y_j^{(n)}),\ i=1,2,\ldots,s,\\
				&y_{n+1}=y_n+h\sum\limits_{j=1}^sb_j\,f(t_n+c_jh,\,y_j^{(n)}),\ n\geq0.
			\end{cases}
		\end{equation}
		This numerical method is frequently utilized for ordinary differential equations (ODEs) structured as $y'(t)=f(t,\,y(t))$ where $t>t_0$, subject to the initial condition $y(t_0)=y_0$. Then, by adapting the method (\ref{2.9}) to the DIDAEs (\ref{2.1}), the following discretisation scheme is obtained:
		
		\begin{equation}\label{2.10}
			\begin{cases}	&u_i^{(n)}=u_{n}+h\sum\limits_{j=1}^{s}a_{ij}\,f(t_j^{(n)},\,u_j^{(n)},\,p_j^{(n)}),\  i=1,2,\cdots,s,\\&u_{n+1}=u_{n}+h\sum\limits_{j=1}^{s}b_{j}\,f(t_j^{(n)},\,u_j^{(n)},\,p_j^{(n)}),\  n=0,1,\cdots,\\&v_{n+1}=g(t_{n+1},\,u_{n+1},\,l_{n+1}),\  n=0,1,\cdots.
			\end{cases}
		\end{equation}
		Define the time step size as $h = \tau/m$, where $m$ being a prescribed positive integer. The usual assumptions are \( \sum\limits_{j=1}^{s} b_j = 1 \).  The discrete time can be defined as \(t_n = t_0 + n h \), as well as \( t_j^{(n)} \) expressed as \( t_n + c_j h \), where \( c_i \in [0,1],\,\forall\  i = 1,2,\dots,s \).  The arguments \( u_n,\ v_n \) approximate \( u(t_n),\  v(t_n) \), respectively. The argument $u_j^{(n)}$ represents  an approximation to $u(t_n+c_jh)$, and the parameter $p_j^{(n)}$ is an approximation to  $\int_{t_j^{(n-m)}}^{t_j^{(n)}}K_{1}(t_j^{(n)},\,\theta,\,u(\theta),\,v(\theta))d\theta$  derived from the compound quadrature formula (CQ formula)
		\begin{equation}\label{2.11}
			p_j^{(n)}=h\sum\limits_{q=0}^m\alpha_q\,K_1(t_j^{(n)},\,t_j^{(n-q)},\,u_j^{(n-q)},\,v_j^{(n-q)}),\  j=1,2,\ldots,s,
		\end{equation}
		where $v_j^{(n)}$ approximates $g(t_j^{(n)},\,u_j^{(n)},\,l_j^{(n)})$, in which $l_j^{(n)}$ is also obtained by CQ formula
		\begin{equation}\label{2.12}
			l_j^{(n)}=h\sum\limits_{q=0}^m\beta_q\,K_2(t_j^{(n)},\,t_j^{(n-q)},\,u_j^{(n-q)},\,v_j^{(n-q)}),\  j=1,2,\ldots,s,
		\end{equation}
		with weights $\{\alpha_q\}$ and $\{\beta_q\}$ that are not dependent on the variable $m$. 
		\( l_n \)   is an approximation of the integral \( \int_{t_{n-m}}^{t_n} K_2(t_n,\, \theta,\, u(\theta),\, v(\theta)) \, d\theta \) and is computed by  CQ formula
		\begin{equation}\label{l_n}
		l_n=h\sum\limits_{q=0}^m\gamma_q\,K_2(t_n,\,t_{n-q},\,u_{n-q},\,v_{n-q}),
		\end{equation}
		where  $\{\gamma_q\}$ are not dependent on the variable $m$. 
		Specifically, the initial conditions satisfy \( u_0 = \psi(t_0) \) and \( v_0 = \varphi(t_0) \).
	
		In the  following  steps, we assume the presence of a constant \(\mu > 0\) 
		in order that the coefficients of the compound quadrature rules (\ref{2.11}) and (\ref{2.12}) fulfill the necessary conditions:
		
		\begin{equation}\label{2.13}
			\begin{aligned}
				&h\sqrt{(m+1)\sum_{q=0}^m|\alpha_q|^2}<\mu, \ mh=\tau,\\&h\sqrt{(m+1)\sum_{q=0}^m|\beta_q|^2}<\mu,\ mh=\tau.
			\end{aligned}
		\end{equation}
		Method (\ref{2.10}) with (\ref{2.11}) and (\ref{2.12}) will further be called CQRK methods.
		\begin{remark}
			In this work, we focus on implicit Runge-Kutta methods for solving nonlinear DIDAEs. This choice is motivated by the fact that such systems often exhibit strong stiffness due to the interplay of algebraic constraints, delay terms, and integral components. Implicit Runge-Kutta methods, especially $A$--stable schemes, are well-suited for handling stiffness while preserving stability properties over long integration intervals. In contrast, explicit methods typically suffer from severe step size restrictions and stability degradation when applied to nonlinear DIDAEs.
		\end{remark}
		\section{Introductory concepts and basic lemmas}
		This section, we introduce several definitions and lemmas that are crucial for obtaining the main result outlined below.
		
		\begin{definition}[see \cite{Burrage1980}]
			Let \( k \) and \( l \) be real constants. A Runge-Kutta method \( (A,\,b^\mathsf{T},\,c) \) is called \((k,\,l)\)--algebraically stable if there exists a diagonal matrix \( D = \mathrm{diag}(d_1,\,d_2,\, \dots,\,d_s) \) with non-negative entries such that the matrix \( \mathcal{M} = [m_{ij}] \) is positive semi-definite, where
			\[
			\mathcal{M} = \begin{pmatrix}
				k - 1 - 2l e^\mathsf{T} D e & e^\mathsf{T} D - b^\mathsf{T} - 2l e^\mathsf{T} D A \\
				D e - b - 2l A^\mathsf{T} D e & D A + A^\mathsf{T} D - b b^\mathsf{T} - 2l A^\mathsf{T} D A
			\end{pmatrix},
			\]
			and \( e = [1,\, 1,\, \dots,\, 1]^\mathsf{T} \). Particularly, when \( k = 1 \) and \( l = 0 \), the method is called algebraically stable.
		\end{definition}
		
		Initially, we present the following notation and conventions:
		\begin{gather*}
			w_{n}=u_{n}-\tilde{u}_{n},\quad W_{i}^{(n)}=u_{i}^{(n)}-\tilde{u}_{i}^{(n)},\\
			r_n=v_n-\tilde{v}_n,\quad v_j^{(n)}=g(t_j^{(n)},\,u_j^{(n)},\,l_j^{(n)}),\quad \tilde{v}_j^{(n)}=g(t_j^{(n)},\,\tilde{u}_j^{(n)},\,\tilde{l}_j^{(n)}),\quad R_j^{(n)}=v_j^{(n)}-\tilde{v}_j^{(n)},\\
			Q_{i}^{(n)}=f(t_{n}+c_{i}h,\,u_{i}^{(n)},\,p_i^{(n)})-f(t_{n}+c_{i}h,\,\tilde{u}_{i}^{(n)},\,\tilde{p}_i^{(n)}).
		\end{gather*}
		Then it follows from (\ref{2.10}) that
		\begin{equation}\label{3.1}
			W_i^{(n)}=w_n+h\sum_{j=1}^sa_{ij}\,Q_j^{(n)},\ i=1,2,\cdots,s,
		\end{equation}
		\begin{equation}\label{3.2}
			w_{n+1}=w_n+h\sum_{j=1}^sb_j\,Q_j^{(n)}.
		\end{equation}

		\begin{definition}
			The CQRK methods are called global stable if there exist positive constants  $\mathsf{H}_1>0$ and $\mathsf{H}_2>0$, which depend only on $L_i\,(i=1,\,2,\,\cdots,\,7),\ \alpha,\ \tau$ and the methods, such that the following conditions hold:
			\begin{equation}
				\begin{aligned}
					&\|w_n\|\leq\mathsf{H}_1\max_{t_0-\tau\leq t\leq t_0}\{\|\psi(t)-\tilde{\psi}(t)\|,\ \|\varphi(t)-\tilde{\varphi}(t)\|\},\ \forall n \geq 1,\\
					&\|r_n\|\leq\mathsf{H}_2\max_{t_0-\tau\leq t\leq t_0}\{\|\psi(t)-\tilde{\psi}(t)\|,\ \|\varphi(t)-\tilde{\varphi}(t)\|\},\ \forall n \geq 1.
				\end{aligned}
			\end{equation} 
		\end{definition}
		Global stability means that the perturbations  in the numerical solution of CQRK methods are directly governed by the problem's initial perturbation. A sufficiently small initial perturbation leads to a correspondingly small perturbation in the numerical solution.
		\begin{definition}
			The CQRK methods are called asymptotically stable if 
			\begin{equation}
				\lim\limits_{n\rightarrow\infty}\|w_n\|=0,\quad \lim\limits_{n\rightarrow\infty}\|r_n\|=0.
			\end{equation} 
		\end{definition}
		Asymptotic stability of the CQRK methods guarantees that any small perturbations introduced into the numerical solution will decay and asymptotically vanish as the time step progresses to infinity, provided the time step size satisfies the stability condition.
		
		The following two lemmas are of significance for the purpose of presenting the stability analysis.
		\begin{lemma}[see \cite{lemma2}]\label{lemma3.1}
			Suppose that $\{\chi_i\}^n_{i=0}$ is an arbitrary sequence of non-negative real numbers. Thus, the inequality below holds
			\begin{equation}
				\sum_{i=0}^n\sum_{j=0}^m\chi_{i-j}\leq(m+1)\sum_{i=0}^n\chi_i+\frac{m(m+1)}{2}\max_{-m\leq q\leq-1}\{\chi_q\},\quad \forall n,m\geq0.
			\end{equation}
		\end{lemma}
		\begin{lemma}\label{lemma3.2}
			Suppose that a $(k,l)$--algebraically stable Runge-Kutta method  $(A,\,b^\mathsf{T},\,c)$  is utilized for solving problem  \eqref{2.1} and  its perturbed counterpart  \eqref{2.8}, where both problems satisfy condition \eqref{2.2}, and suppose the compound quadrature rules \eqref{2.11} and \eqref{2.12} satisfy condition \eqref{2.13}. Consequently, the following inequality holds
			\begin{equation}\label{3.6}
				\begin{aligned}
					\|w_{n+1}\|^2\leq &k\|w_n\|^2+\sum_{j=1}^sd_j((2h\alpha+hL_1-2l)\|W_j^{(n)}\|^2\\&+\frac{2\mu^2hL_1}{m+1}(\sum\limits_{q=0}^m(L_4^2\|W_j^{(n-q)}\|^2+L_5^2\|R_j^{(n-q)}\|^2)).
				\end{aligned}
			\end{equation}
		\end{lemma}
		\begin{proof}
			
			It follows from the $(k,l)$--algebraic stability property of the method that \cite{Burrage1980}
			\begin{equation}
				\|w_{n+1}\|^2-k\|w_n\|^2-2\sum_{j=1}^sd_jRe\langle W_j^{(n)},hQ_j^{(n)}-lW_j^{(n)}\rangle=-\sum_{i=1}^{s+1}\sum_{j=1}^{s+1}m_{ij}\langle\theta_i,\theta_j\rangle,
			\end{equation}
			where $\mathcal{M}=[m_{ij}],\ \theta_1=w_n,\ \theta_{j+1}=hQ_j^{(n)},\ j=1,2,\cdots,s.$ Hence, one has 
			\begin{equation}\label{3.8}
				\|w_{n+1}\|^2\leq k\|w_n\|^2+2\sum_{j=1}^sd_j\Re\langle W_j^{(n)},hQ_j^{(n)}-lW_j^{(n)}\rangle.
			\end{equation}
			From \eqref{2.6}, another result follows
			\begin{equation}
				\begin{aligned}\label{3.9}
					2\Re\langle W_j^{(n)},hQ_j^{(n)}\rangle
					=&2h( \Re\langle u_i^{(n)}-\tilde{ u}_i^{(n)},f(t_j^{(n)},\,u_j^{(n)},\,p_j^{(n)})-f(t_j^{(n)},\,\tilde{u}_j^{(n)},\,p_j^{(n)})\rangle \\
					&+\Re\langle u_i^{(n)}-\tilde{ u}_i^{(n)},f(t_j^{(n)},\,\tilde{u}_j^{(n)},\,p_j^{(n)})-f(t_j^{(n)},\,\tilde{u}_j^{(n)},\,\tilde{p}_j^{(n)})\rangle )\\
					\leq&2h\alpha\|W_j^{(n)}\|^2+2h\|W_j^{(n)}\|\|f(t_j^{(n)},\,\tilde{u}_j^{(n)},\,p_j^{(n)})-f(t_j^{(n)},\,\tilde{u}_j^{(n)},\,\tilde{p}_j^{(n)})\|\\
					\leq&2h\alpha\|W_j^{(n)}\|^2+2h L_1\|W_j^{(n)}\|\|p_j^{(n)}-\tilde{p}_j^{(n)}\|\\
					\leq&2h\alpha\|W_j^{(n)}\|^2+hL_1(\|W_j^{(n)}\|^2+\|p_j^{(n)}-\tilde{p}_j^{(n)}\|^2),
				\end{aligned}
			\end{equation}
			where the latter is derived by applying the inequality  $2ab\leq a^2+b^2$, for all real numbers $a$ and $b$. Inserting \eqref{3.9} into \eqref{3.8}, we have
			\begin{equation}\label{3.10}
				\|w_{n+1}\|^2\leq k\|w_n\|^2+(2h\alpha+hL_1-2l)\sum\limits_{j=1}^{s}d_j\|W_j^{(n)}\|^2+hL_1\sum\limits_{j=1}^{s}d_j\|p_j^{(n)}-\tilde{p}_j^{(n)}\|^2.
			\end{equation}
			By conditions \eqref{2.4} and \eqref{2.11}, we have
			\begin{equation}\label{3.11}
				\begin{aligned}
					&\lVert p_j^{(n)}-\tilde{p}_j^{(n)}\|^2
					\\=&\|h\sum\limits_{q=0}^m\alpha_qK_1(t_j^{(n)},\,t_j^{(n-q)},\,u_j^{(n-q)},\,v_j^{(n-q)})-h\sum\limits_{q=0}^m\alpha_qK_1(t_j^{(n)},\,t_j^{(n-q)},\,\tilde{u}_j^{(n-q)},\,\tilde{v}_j^{(n-q)})\rVert^2\\
					=&\|h\sum\limits_{q=0}^m\alpha_q(K_1(t_j^{(n)},\,t_j^{(n-q)},\,u_j^{(n-q)},\,v_j^{(n-q)})-K_1(t_j^{(n)},\,t_j^{(n-q)},\,\tilde{u}_j^{(n-q)},\,\tilde{v}_j^{(n-q)}))\|^2	\\
					\leq& h^2(\sum\limits_{q=0}^m|\alpha_q|^2)(\sum\limits_{q=0}^m\|K_1(t_j^{(n)},\,t_j^{(n-q)},\,u_j^{(n-q)},\,v_j^{(n-q)})-K_1(t_j^{(n)},\,t_j^{(n-q)},\,\tilde{u}_j^{(n-q)},\,\tilde{v}_j^{(n-q)})\|^2)\\
					\leq& h^2(\sum\limits_{q=0}^m|\alpha_q|^2)(\sum\limits_{q=0}^m(L_4\|W_j^{(n-q)}\|+L_5\|R_j^{(n-q)}\|)^2)\\
					\leq &	2h^2(\sum\limits_{q=0}^m|\alpha_q|^2)(\sum\limits_{q=0}^m(L_4^2\|W_j^{(n-q)}\|^2+L_5^2\|R_j^{(n-q)}\|^2))\\
					\leq&\frac{2\mu^2}{m+1}	
					\sum\limits_{q=0}^m(L_4^2\|W_j^{(n-q)}\|^2+L_5^2\|R_j^{(n-q)}\|^2).
				\end{aligned}
			\end{equation}
			Inserting \eqref{3.11} into \eqref{3.10}, we have \eqref{3.6} and  finalize the lemma’s proof.
			
		\end{proof}
		
		\section{Stability of Runge-Kutta methods for solving DIDAEs}
		This section examines the global and asymptotic stability properties of CQRK methods. 
		\begin{theorem}\label{thero1}
			Assuming the given RK method \eqref{2.9} is $(k,l)$--algebraically stable with respect to a diagonal matrix with non-negative entries $D = \mathrm{diag}(d_i)_{i=1}^s\in \mathbb{R}^{s\times s}$, with $k\in (0,1]$. Further assume that compound quadrature formulas given by  \eqref{2.11} and \eqref{2.12} satisfy condition \eqref{2.13}. Then, the CQRK methods are globally stable, whenever
			\begin{gather}	\label{4.1}h(2\alpha+L_1+2\mu^2L_1L_4^2+\frac{2\mu^2L_1L_5^2(2L_2^2+4\mu^2L_3^2L_6^2)}{1-4\mu^2L_3^2L_7^2})<2l,\\
				\label{4.2}	2\gamma\tau L_3L_7<1,\qquad 4\mu^2L_3^2L_7^2<1,
			\end{gather}
			where $\gamma = \max\limits_{0\leq q\leq m}|\gamma_q|$.
		\end{theorem}
		\begin{proof}
			Since $0<k\leq 1$, using induction on \eqref{3.6}, we have
			\begin{equation}\label{4.3}
				\begin{aligned}
					\|w_{n+1}\|^2\leq&\|w_0\|^2+(2h\alpha+hL_1-2l)\sum_{i=0}^n\sum_{j=1}^sd_j\|W_j^{(i)}\|^2\\
					&+\frac{2\mu^2hL_1}{m+1}\sum_{j=1}^sd_j\sum_{i=0}^n\sum\limits_{q=0}^m(L_4^2\|W_j^{(i-q)}\|^2+L_5^2\|R_j^{(i-q)}\|^2).
				\end{aligned}
			\end{equation}
			It follows from Lemma \ref{lemma3.1} and condition $mh=\tau$ that
			\begin{equation}\label{4.4}
				\begin{aligned}
					\|w_{n+1}\|^2\leq&\|w_0\|^2+(2h\alpha+hL_1-2l)\sum_{j=1}^sd_j\sum_{i=0}^n\|W_j^{(i)}\|^2\\
					&+2\mu^2hL_1L_4^2\sum_{j=1}^sd_j\|W_j^{(i)}\|^2+\mu^2\tau L_1L_4^2\sum_{j=1}^sd_j\max_{-m\leq i\leq-1}\{\|W_j^{(i)}\|^2\}\\
					&+2\mu^2hL_1L_5^2\sum_{j=1}^sd_j\sum_{i=0}^n\|R_j^{(i)}\|^2+\mu^2\tau L_1L_5^2\sum_{j=1}^sd_j\max_{-m\leq i\leq-1}\{\|R_j^{(i)}\|^2\}\\
					=&\|w_0\|^2+(2h\alpha+hL_1-2l+2\mu^2hL_1L_4^2)\sum_{j=1}^sd_j\sum_{i=0}^n\|W_j^{(i)}\|^2\\
					&+\mu^2\tau L_1L_4^2\sum_{j=1}^sd_j\max_{-m\leq i\leq-1}\{\|W_j^{(i)}\|^2\}+2\mu^2hL_1L_5^2\sum_{j=1}^sd_j\sum_{i=0}^n\|R_j^{(i)}\|^2\\
					&+\mu^2\tau L_1L_5^2\sum_{j=1}^sd_j\max_{-m\leq i\leq-1}\{\|R_j^{(i)}\|^2\}.
				\end{aligned}
			\end{equation}
			By \eqref{2.3}, we have
			\begin{equation}\label{4.5}
				\begin{aligned}
					\sum_{i=0}^n\|R_j^{(i)}\|^2&=\sum_{i=0}^n\|g(t_j^{(i)},\,u_{j}^{(i)},\,l_{j}^{(i)})-g(t_j^{(i)},\,\tilde{u}_{j}^{(i)},\,\tilde{l}_{j}^{(i)})\|^2\\&\leq 2L_2^2\sum_{i=0}^n\|u_{j}^{(i)}-\tilde{u}_{j}^{(i)}\|^2+2L_3^2 \sum_{i=0}^n\|l_{j}^{(i)}-\tilde{l}_{j}^{(i)}\|^2.
				\end{aligned}
			\end{equation}
			With conditions \eqref{2.5}, \eqref{2.12} and \eqref{2.13}, we have
			\begin{equation}\label{4.6}
				\begin{aligned}
					&\|l_{j}^{(i)}-\tilde{l}_{j}^{(i)}\|^2\\=&\|h\sum\limits_{q=0}^m\beta_q\,K_2(t_j^{(i)},\,t_j^{(i-q)},\,u_j^{(i-q)},\,v_j^{(i-q)})-h\sum\limits_{q=0}^m\beta_q\,K_2(t_j^{(i)},\,t_j^{(i-q)},\,\tilde{u}_j^{(i-q)},\,\tilde{v}_j^{(i-q)})\|^2\\
					=&\|h\sum\limits_{q=0}^m\beta_q(K_2(t_j^{(i)},\,t_j^{(i-q)},\,u_j^{(i-q)},\,v_j^{(i-q)})-K_2(t_j^{(i)},\,t_j^{(i-q)},\,\tilde{u}_j^{(i-q)},\,\tilde{v}_j^{(i-q)}))\|^2\\
					\leq& 2h^2(\sum\limits_{q=0}^m|\beta_q|^2)(\sum\limits_{q=0}^m(L_6^2\|W_j^{(i-q)}\|^2+L_7^2\|R_j^{(i-q)}\|^2))\\
					\leq& \frac{2\mu^2}{m+1}\sum\limits_{q=0}^m(L_6^2\|W_j^{(i-q)}\|^2+L_7^2\|R_j^{(i-q)}\|^2).
				\end{aligned}
			\end{equation}
			Embedding \eqref{4.6} into \eqref{4.5} yields
			\begin{equation}\label{4.7}
				\begin{aligned}
					\sum_{i=0}^n\|R_j^{(i)}\|^2&\leq 2L_2^2\sum_{i=0}^n\|W_j^{(i)}\|^2+\frac{4\mu^2L_3^2}{m+1}\sum_{i=0}^n\sum\limits_{q=0}^m(L_6^2\|W_j^{(i-q)}\|^2+L_7^2\|R_j^{(i-q)}\|^2).
				\end{aligned}
			\end{equation}
			Applying lemma \ref{lemma3.1} to \eqref{4.7} shows
			\begin{equation}\label{4.8}
				\begin{aligned}
					\sum_{i=0}^n\|R_j^{(i)}\|^2
					\leq&2L_2^2\sum_{i=0}^n\|W_j^{(i)}\|^2+\frac{4\mu^2L_3^2}{m+1}(L_6^2(m+1)\sum_{i=0}^n\|W_j^{(i)}\|^2+\frac{L_6^2m(m+1)}{2}\max_{-m\leq i\leq-1}\\&\times \{\|W_j^{(i)}\|^2\}
					+L_7^2(m+1)\sum_{i=0}^n\|R_j^{(i)}\|^2+\frac{L_7^2m(m+1)}{2}\max_{-m\leq i\leq-1}\{\|R_j^{(i)}\|^2\})\\
					=& (2L_2^2+4\mu^2L_3^2L_6^2)\sum_{i=0}^n\|W_j^{(i)}\|^2+2\mu^2L^2_3L^2_6m\max_{-m\leq i\leq-1}\{\|W_j^{(i)}\|^2\}\\
					&+4\mu^2L_3^2L_7^2\sum_{i=0}^n\|R_j^{(i)}\|^2+2\mu^2L^2_3L^2_7m\max_{-m\leq i\leq-1}\{\|R_j^{(i)}\|^2\}.
				\end{aligned}
			\end{equation}
			Bound \eqref{4.8} therefore implies
			\begin{equation}\label{4.9}
				\begin{aligned}
					\sum_{i=0}^n\|R_j^{(i)}\|^2\leq&\frac{2L_2^2+4\mu^2L_3^2L_6^2}{1-4\mu^2L_3^2L_7^2}\sum_{i=0}^n\|W_j^{(i)}\|^2+\frac{2\mu^2mL^2_3L^2_6}{1-4\mu^2L_3^2L_7^2}\max_{-m\leq i\leq-1}\{\|W_j^{(i)}\|^2\}\\
					&+\frac{2\mu^2mL^2_3L^2_7}{1-4\mu^2L_3^2L_7^2}\max_{-m\leq i\leq-1}\{\|R_j^{(i)}\|^2\}.
				\end{aligned}
			\end{equation}
			By inserting equation \eqref{4.9} into the expression for $\|w_{n+1}\|^2$, we derive an additional upper limit for  $\|w_{n+1}\|^2$:
			\begin{equation}\label{4.10}
				\begin{aligned}
					&\|w_0\|^2+(2h\alpha+hL_1-2l+2\mu^2hL_1L_4^2+\frac{2\mu^2hL_1L_5^2(2L_2^2+4\mu^2L_3^2L_6^2)}{1-4\mu^2L_3^2L_7^2})\\
					&\times \sum_{j=1}^sd_j\sum_{i=0}^n\|W_j^{(i)}\|^2+(\mu^2\tau L_1L_4^2+\frac{4\mu^4\tau L_1L_3^2L_5^2L_6^2}{1-4\mu^2L_3^2L_7^2})\sum_{j=1}^sd_j\max_{-m\leq i\leq-1}\{\|W_j^{(i)}\|^2\}\\
					&+(\mu^2\tau L_1L_5^2+\frac{4\mu^4\tau L_1L_3^2L_5^2L_7^2}{1-4\mu^2L_3^2L_7^2})\sum_{j=1}^sd_j\max_{-m\leq i\leq-1}\{\|R_j^{(i)}\|^2\}.
				\end{aligned}
			\end{equation}
			This step refines the estimation of the bound, providing a tighter approximation  to the value of $\|w_{n+1}\|^2$.	Since by
			\[
			h(2\alpha+L_1+2\mu^2L_1L_4^2+\frac{2\mu^2L_1L_5^2(2L_2^2+4\mu^2L_3^2L_6^2)}{1-4\mu^2L_3^2L_7^2})<2l,
			\]
			hence \eqref{4.10} implies that 
			\begin{equation}\label{4.11}
				\begin{aligned}
					\|w_{n+1}\|^2
					\leq &\|w_0\|^2+(\mu^2\tau L_1L_4^2+\frac{4\mu^4\tau L_1L_3^2L_5^2L_6^2}{1-4\mu^2L_3^2L_7^2})\sum_{j=1}^sd_j\max_{-m\leq i\leq-1}\{\|W_j^{(i)}\|^2\}
					\\&+(\mu^2\tau L_2L_5^2+\frac{4\mu^4\tau L_1L_3^2L_5^2L_7^2}{1-4\mu^2L_3^2L_7^2}) \sum_{j=1}^sd_j\max_{-m\leq i\leq-1}\{\|R_j^{(i)}\|^2\}
					.
				\end{aligned}
			\end{equation}
			Therefore there exists a constant $\mathsf{H}_1$, which depends only on $L_i\,(i=1,\,2,\,\cdots,\,7),\ \alpha,\ \tau$ and the methods, such that the following inequality holds
			\begin{equation}\label{4.12}
				\|w_n\|\leq\mathsf{H}_1\max_{t_0-\tau\leq t\leq t_0}\{\|\psi(t)-\tilde{\psi}(t)\|,\ \|\varphi(t)-\tilde{\varphi}(t)\|\}.
			\end{equation}
			To simplify notation, let $H_1:=\mathsf{H}_1\max\limits_{t_0-\tau\leq t\leq t_0}\{\|\psi(t)-\tilde{\psi}(t)\|,\ \|\varphi(t)-\tilde{\varphi}(t)\|\}$. 
			For the algebraic equation, we have
			\begin{equation}\label{4.13}
				\begin{aligned}
					\|r_n\|=&\|g(t_n,u_n,l_n)-g(t_n,\tilde{u}_n,\tilde{l}_n)\|\\
					\leq &L_2\|w_n\|+L_3\|h\sum\limits_{q=0}^m\gamma_qK_2(t_n,t_{n-q},u_{n-q},v_{n-q})
					\\&-h\sum\limits_{q=0}^m\gamma_qK_2(t_n,t_{n-q},\tilde{u}_{n-q},\tilde{v}_{n-q})\|\\
					\leq &L_2\|w_n\|+\gamma hL_3\sum_{q=0}^{m}(L_6\|w_{n-q}\|+L_7\|r_{n-q}\|)\\
					\leq & (L_2+2\gamma\tau L_3L_6)H_1+\gamma h L_3L_7\sum_{q=0}^{m}\|r_{n-q}\|.
				\end{aligned}
			\end{equation}
			For any $n\geq m$, we consider two cases. Firstly, if $\max\limits_{0\leq q\leq m}\|r_{n-q}\|=\|r_{n}\|$, we have
			\[
			\|r_n\|\leq (L_2+2\gamma \tau L_3L_6)H_1+2\gamma\tau L_3L_7\|r_{n}\|,
			\]	
			and therefore 
			\begin{equation}\label{4.14}
				\|r_n\|\leq \frac{L_2+2\gamma \tau L_3L_6}{1-2\gamma\tau L_3L_7}H_1.
			\end{equation}
			Secondly, suppose there exist integers \( 0 < \lambda_i \leq m \) for \( i = 1,\, \dots,\, m \), with the property that $\max\limits_{0\leq q\leq m}\|r_{n-q}\|=\|r_{n-\lambda_i}\|$, then there exists a constant $\omega>0$ that satisfies $-m\leq n-\sum\limits_{i=0}^{\omega}\lambda_i<-1$, hence, we have
			\begin{equation}\label{4.15}
				\begin{aligned}
					\|r_n\|\leq& (L_2+2\gamma \tau L_3L_6)H_1+2\gamma\tau L_3L_7\|r_{n-\lambda_i}\|\\
					\leq& \sum_{q=0}^{\omega}(2\gamma\tau L_2L_7)^q(L_2+2\gamma \tau L_3L_6)H_1+(2\gamma\tau L_3L_7)^{\omega}\|r_{n-\sum_{i=0}^{\omega}\lambda_i}\|.
				\end{aligned}
			\end{equation}
			Combining this with \eqref{4.14} shows that there exists a constant $\mathsf{H}_2$, which depends only on $L_i (i=1,\,2,\,\cdots,\,7),\ \alpha,\ \tau$ and the methods, such that the following inequality holds
			\begin{equation}\label{4.16}
				\|r_n\|\leq\mathsf{H}_2\max_{t_0-\tau\leq t\leq t_0}\{\|\psi(t)-\tilde{\psi}(t)\|,\ \|\varphi(t)-\tilde{\varphi}(t)\|\}.
			\end{equation}
			This, together with \eqref{4.12}, implies that the method is globally stable.
		\end{proof}
		
		In the following discussion, the concept of asymptotic stability will be examined. The subsequent theorem will be utilised in this endeavour.
		
		\begin{theorem}\label{theorem2}
			Suppose that the fundamental RK method \eqref{2.9} is algebraically stable for a diagonal matrix with positive entries $D > 0$ and satisfies $\left|1 - b^\mathrm{T} A^{-1} e\right| < 1$ with $\det A \neq 0$. Additionally, suppose the quadrature formulas \eqref{2.11}--\eqref{2.12} meet the condition \eqref{2.13}. Then, the CQRK methods are asymptotically stable provided that,
			\begin{gather}	\label{4.17}2\alpha+L_1+2\mu^2L_1L_4^2+\frac{2\mu^2L_1L_5^2(2L_2^2+4\mu^2L_3^2L_6^2)}{1-4\mu^2L_3^2L_7^2}<0,\\
				\label{4.18}	2\gamma\tau L_3L_7<1,\qquad 4\mu^2L_3^2L_7^2<1,
			\end{gather}
			where $\gamma = \max\limits_{0\leq q\leq m}|\gamma_q|$.
		\end{theorem}
		\begin{proof}
			It follow from \eqref{4.10} that 
			\begin{equation}\label{4.19}
				\lim_{n\rightarrow \infty}\|W_j^{(n)}\|=0,\ j=1,\cdots,s.
			\end{equation}	
			Since $\det A\neq0$, the matrix  $A$ is non-singular. Let $G=[g_{ij}]=A^{-1}$. From equations \eqref{3.1}--\eqref{3.2}, we derive the relationship
			\[w_{n+1}=(1-b^\mathsf{T}A^{-1}e)w_n+\sum\limits_{i=1}^{s}\sum\limits_{j=1}^{s}g_{ij}\,b_i\,W_j^{(n)}.\]
			Therefore from \eqref{4.19} and $|1-b^\mathsf{T}A^{-1}e|<1$, it is straightforward to obtain
			\begin{equation}
				\lim\limits_{n\rightarrow\infty}\|w_n\|=0.
			\end{equation}
			From \eqref{4.13} and $mh=\tau$ we obtain
			\begin{equation}\label{4.23}
				\begin{aligned}
					\lim_{n\to\infty}\|r_n\|\leq& \lim\limits_{n\rightarrow\infty}(L_2+2\beta \tau L_3L_6)\|w_n\|+\lim\limits_{n\rightarrow\infty}h\gamma L_3L_7\sum_{q=0}^{m}\|r_{n-q}\|\\
					=&\lim\limits_{n\rightarrow\infty}2\gamma\tau L_3L_7\|r_{n}\|.
				\end{aligned}
			\end{equation}
			For $2\gamma\tau L_3L_7<1$, we have 
			\begin{equation}
				\lim\limits_{n\rightarrow\infty}\|r_n\|=0.
			\end{equation}
			Thus, the proof of Theorem \ref{theorem2} is complete.
		\end{proof}
		\section{Numerical examples}
		\begin{example}
			Analyze the initial value problem for partial DIDAEs
			\begin{equation}\label{6.1}
				\begin{cases}
					&\frac{\partial y(s,t)}{\partial t}=\frac{\partial^2y(s,t)}{\partial s^2}+\int_{t-\frac{\pi}{2}}^t2y(s,\theta )z(s,\theta)d\theta+f_1(s,t),\quad 0<s<1,\, 0<t,\\&2(y(s,t)+1)z(s,t)+0.5\int_{t-\frac{\pi}{2}}^t\sin \theta\cos(2\theta)y(s,\theta)z(s,\theta)d\theta+f_2(s,t)=0,\quad 0<s<1,\, 0<t,\\&y(s,t)=(s^2-s)\cos(t),\ z(s,t)=(s^2-s)\sin(t),\quad 0<s<1,\, -\frac{\pi}{2}\leq t\leq0,\\&y(0,t)=y(1,t)=z(0,t)=z(1,t)=0,\quad 0\leq t,
				\end{cases}
			\end{equation}
			where
			$$\begin{cases}f_1(s,t)=(s^2-s)\sin(t)-2\cos(t)+(s^2-s)^2\cos(2t),
				\\f_2(s,t)=-(s^2-s)^2\sin(2t)-2(s^2-s)\sin(t)-0.125(s^2-s)^2(\sin(2t)-\cos(2t)).\end{cases}$$
			
			This problem admits a unique exact solution
			
			$$y(s,t)=(s^2-s)\cos(t)\text{ and }z(s,t)=(s^2-s)\sin(t).$$
		Applying the method of lines, equations \eqref{6.1} can be discretized as follows:
			\begin{equation}\label{6.2}
				\begin{cases}
					&\frac{\partial y_i(t)}{\partial t}=\frac{y_{i+1}(t)-2y_i(t)+y_{i-1}(t)}{h_s^2}+\int_{t-\frac{\pi}{2}}^t2y_i(\theta)z_i(\theta)d\theta+f_1(s_i,t),\ 0<t,\ i=1,2,\ldots,N-1,\\
					&2(y_i(t)+1)z_i(t)+0.5\int_{t-\frac{\pi}{2}}^t\sin(\theta)\cos(2\theta)y_i(\theta)z_i(\theta)d\theta+f_2(s_i,t)=0,\ 0<t,\ i=1,2,\ldots,N-1,\\
					&y_i(t)=(s_i^2-s_i)\cos(t),z_i(t)=(s_i^2-s_i)\sin(t),\ -\frac{\pi}{2}\leq t\leq0,\  i=1,2,\ldots,N-1,
					\\
					&y_0(t)=y_N(t)=z_0(t)=z_N(t)=0,\  0\leq t.
				\end{cases}
			\end{equation}
			Here, \(\Delta_s\) denotes the spatial discretization step size, and  \(N_s\) is a given positive integer satisfying \(N_s\Delta_s = 1\). Define the spatial grid points \(s_i = i \Delta_s\) for \(i =  1, \dots, N_s-1\), and note that  \(y_i(t) = y(s_i, t)\) and \(z_i(t) = z(s_i, t)\).
			An equal step size \( h \) is adopted, satisfying \( hN = \pi/2 \), where \( N \) is a given positive integer. We set \( N = 10\text{ and }100 \) and $N_s=10$. It can be verified that equations \eqref{6.2} satisfy conditions \eqref{2.2}--\eqref{2.6} with
			$$\alpha=-4N_s^2\sin^2 \frac{\pi}{2N_s},\ L_1=1,\ L_2= L_3=L_4=L_5=\frac{1}{2},\ L_6=L_7=\frac{1}{4}.$$		
			The given differential equation is discretized using the 2-stage Lobatto \uppercase\expandafter{\romannumeral3}C Runge-Kutta method
			\[
			(A,\,b^\mathsf{T},\,c)=\begin{array}{c|cc}
				0 & 1/2 & -1/2 \\
				1 & 1/2 & 1/2 \\
				\hline
				& 1/2 & 1/2 \\
			\end{array},
			\] and the integral terms are approximated by Compound Simpson’s rule, where $\alpha_0=\alpha_N=\beta_0=\beta_N=\frac{1}{3}$, $\alpha_i\leq\frac{4}{3},\ \beta_i\leq\frac{4}{3}$ for $i = 1, 2, \dots, N-1$, and $\gamma_i\leq \frac{4}{3}$ for $i=0,1,\cdots,N$. We therefore choose $\mu = \frac{5}{2}$and $\gamma = \frac{4}{3}$.  Noting the algebraic stability of 2-stage Lobatto \uppercase\expandafter{\romannumeral3}C method, conditions \eqref{4.1}--\eqref{4.2} and \eqref{4.17}--\eqref{4.18} are satisfied. Consequently, the method is globally and asymptotically stable.
			
			For the exact solution of equations~\eqref{6.2}, the initial values \( \{y_i(0), z_i(0)\} \) are given by:
			\[\{y_i(0)=s_i^2-s_i,\ z_i(0)=0,\quad i=1,2,\cdots,N_s-1.\}\]
			The perturbed initial funcions $\{\tilde{u}_i(0),\tilde{v}_i(0)\}$  are defined as: $$\{\tilde{y}_i(0)=s_i^2-s_i+0.5,\ \tilde{z}_i(0)=0.5,\quad  i=1,2,\cdots,N_s-1.\}$$
			Let \(\{Y_n, Z_n\}\) and \(\{\tilde{Y}_n, \tilde{Z}_n\}\) denote the numerical solutions obtained from \(\{y_i(0), z_i(0)\}\) and \(\{\tilde{y}_i(0), \tilde{z}_i(0)\}\) respectively, where
			 \begin{equation}
				\begin{aligned}
					&Y_n=[y_{1}(t_n),y_{2}(t_n),\cdots,y_{N_s-1}(t_n)],\quad Z(t)=[z_{1}(t_n),z_{2}(t_n),\cdots,z_{N_s-1}(t_n)],\\
					&\tilde{Y}_n=[\tilde{y}_{1}(t_n),\tilde{y}_{2}(t_n),\cdots,\tilde{y}_{N_s-1}(t_n)],\quad \tilde{Z}_n=[\tilde{z}_{1}(t_n),\tilde{z}_{2}(t_n),\cdots,\tilde{z}_{N_s-1}(t_n)].
				\end{aligned}
			 \end{equation}
			 Fig.\ref{fig:example1} shows the resulting disturbance errors under the condition of N=10, and 			fig.\ref{fig:example1_100} displays the resulting disturbance errors for N = 100, where  the calculations of $\|Y_n-\tilde{Y}_n\| \text{ and } \|Z_n-\tilde{Z}_n\|$ are given by
			 \begin{equation*}
			 	\begin{aligned}
			 		\| Y_n - \tilde{Y}_n \| = \sqrt{\Delta s \sum_{i=1}^{N_s - 1} | y_i(t_n) - \widetilde{y}_i(t_n) |^2},\\
			 		\| Z_n - \tilde{Z}_n \| = \sqrt{\Delta s \sum_{i=1}^{N_s - 1} | z_i(t_n) - \widetilde{z}_i(t_n) |^2}.\\
			 	\end{aligned}
			 \end{equation*}
			\begin{figure}[htbp]
				\centering
				\includegraphics{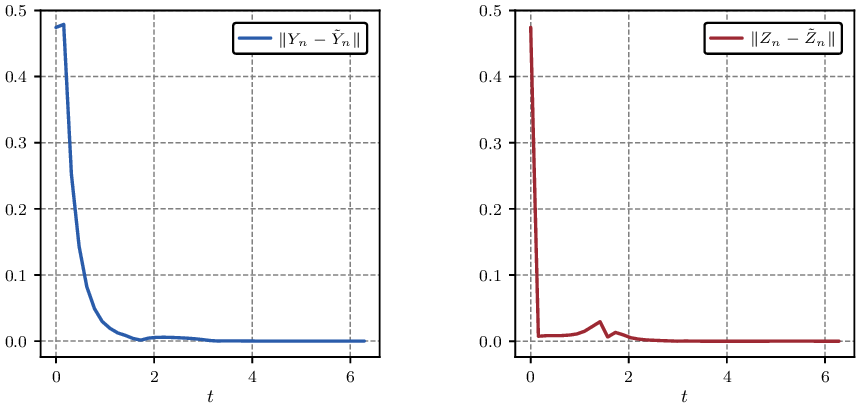}
				\caption{The disturbance errors $\|Y_n-\tilde{Y}_n\| \text{ and } \|Z_n-\tilde{Z}_n\|$ when $N=10$.}
				\label{fig:example1}
			\end{figure}

			 \begin{figure}[htbp]
			 	\centering
			 	\includegraphics{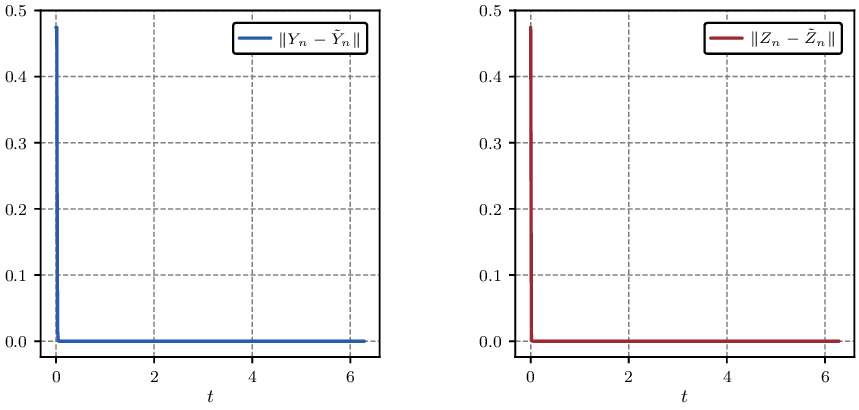}
			 	\caption{The disturbance errors $\|Y_n-\tilde{Y}_n\| \text{ and } \|Z_n-\tilde{Z}_n\|$ when $N=100$.}
			 	\label{fig:example1_100}
			 \end{figure}
			 To  demonstrate the convergence of the method, we define
			 \[
			 E(t)=\|Y_n-\tilde{Y}_n\|_\infty,\ EA(t)=\|Z_n-\tilde{Z}_n\|_\infty
			 \]
			 to represent their difference. Table \ref{tab:combined} lists partial numerical results of $E(t)$ and $EA(t)$ when \(N = 10\) and \(N = 100\) respectively. Note that since we chose the exact solution and its perturbation as the initial conditions, we can conclude from the table \ref{tab:combined} that the method is convergent.
			 \begin{table}[htbp] 
			 	\centering
			 	\caption{When numerical methods are applied to Equations \eqref{6.1}, the differences between the two solutions are \(E(t)\) and \(EA(t)\)} 
			 	\begin{tabular}{ccccccc} 
			 		\toprule 
			 		\multirow{2}{*}{$N$} & \multicolumn{3}{c}{$E(t)$} & \multicolumn{3}{c}{$EA(t)$} \\
			 		\cmidrule(lr){2-4} \cmidrule(lr){5-7}
			 		& $t = \frac{\pi}{2}$ & $t = \frac{3}{2}\pi$ & $t = \frac{5}{2}\pi$ & $t = \frac{\pi}{2}$ & $t = \frac{3}{2}\pi$ & $t = \frac{5}{2}\pi$ \\
			 		\midrule 
			 		10 & 4.9525e-03 & 7.5835e-05 & 1.3489e-06 & 6.0145e-02 & 2.1954e-05 & 1.5814e-05 \\
			 		100 & 6.6841e-06 & 5.4973e-15 & 3.4694e-18 & 1.5299e-06 & 1.5693e-13 & 6.7221e-17 \\
			 		\bottomrule 
			 	\end{tabular}
			 	\label{tab:combined}
			 \end{table}

			 In the comparative experiment, we replaced 2-stage Lobatto \uppercase\expandafter{\romannumeral3}C Runge-Kutta method with the classical 4th--order Runge-Kutta method while keeping all other conditions unchanged. When $N=100$, due to the method's failure to satisfy the $(k,l)$--algebraic stability, the numerical results exhibited divergent characteristics. Figure \eqref{example_1_compare} shows the evolution of the perturbation error $\|Y_n-\tilde{Y}_n\|$ when the classical 4th-order Runge-Kutta method is used to solve Equation \eqref{6.1}. The comparative experimental results demonstrate that the method proposed in this paper demonstrates good stability and accuracy in handling stiff problems. 
			 \begin{figure}[htbp]
			 	\centering
			 	\includegraphics{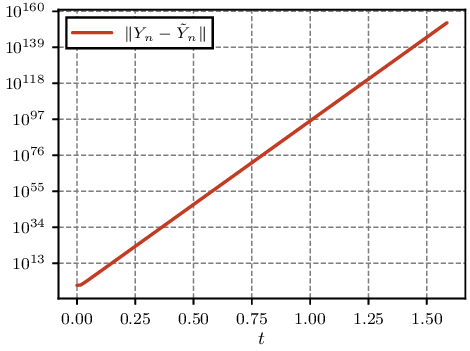}
			 	\caption{The disturbance errors $\|Y_n-\tilde{Y}_n\|$ when $N=100$.}
			 	\label{example_1_compare}
			 \end{figure}
			 
			 These numerical results confirm the correctness of the stability results obtained in this paper.  
		\end{example}
		\begin{example}
			Consider the initial value problem for DIDAEs
			\begin{equation}\label{6.3}
				\begin{cases}&u_{1}^{\prime}(t)=t^{2}e^{-t}-50u_1(t)+u_{2}(t)\int_{t-1}^{t}e^{\theta-t}(u_{2}(\theta)+v_{1}(\theta))d\theta+f_{1}(t),\quad t> 0,\\&u_{2}^{\prime}(t)=1+\sin t^{2}-50u_2(t)+u_{1}(t)\int_{t-1}^{t}e^{\theta-t}(u_{1}(\theta)-\-v_{2}(\theta))d\theta+f_{2}(t),\quad t> 0,\\&v_{1}(t)=-0.1u_{2}(t)+0.25\int_{t-1}^{t}\sin(t-\theta)(u_{2}(\theta)-0.25v_{1}(\theta))d\theta+g_{1}(t),\quad t> 0,\\&v_{2}(t)=0.2u_{1}(t)+0.25\int_{t-1}^{t}\cos(t-\theta)(u_{1}(\theta)+0.25v_{2}(\theta))d\theta+g_{2}(t),\quad t> 0,\\&u_{1}(t)=\cos (t)e^{-t},u_{2}(t)=\sin (t)e^{-t},\quad-1\leq t\leq0,\\&v_{1}(t)=(1-t)e^{-t},v_{2}(t)=(1+t)e^{-t},\quad-1\leq t\leq0.\end{cases}
			\end{equation}
			Here, \(f_1(t)\), \(f_2(t)\), \(g_1(t)\), and \(g_2(t)\) are specifically constructed functions for which the differential system \eqref{6.3} admits the exact solutions \(u(t)=e^{-t} \left[ {\begin{array}{*{20}{c}}
					{{\cos(t)}}\\
					{{\sin(t)}}
			\end{array}} \right] \) and \(v(t)=e^{-t}\left[ {\begin{array}{*{20}{c}}
			{{1-t}}\\
			{{1+t}}
		\end{array}} \right] \). The system \eqref{6.3} satisfies conditions \eqref{2.2}--\eqref{2.7} with 
			$$\alpha =-50,\ L_1=1,\ L_2=\frac{1}{5},\ L_3=\frac{1}{4},\ L_4=L_5=L_6=2,\ L_7=\frac{1}{2}.$$	
			To examine the asymptotic and global stability of the method proposed in this paper, similarly, by discretizing the differential equation using 2-stage Lobatto \uppercase\expandafter{\romannumeral3}C Runge-Kutta method and approximating the integral terms with the Compound Simpson's rule. Step sizes of \(h = 0.1\) and \(h = 0.01\) are adopted, and  we can verify that the method satisfies the conditions of Theorem \ref{thero1} and \ref{theorem2} with \(\gamma=\frac{4}{3}, \mu=\frac{5}{2}\). Perturbed initial functions are considered with
			\[\begin{cases}
				&\tilde{u}_1(t)=\cos(t)(e^{-t}+0.5),\ \tilde{u}_2(t)=\sin(t)(e^{-t}+0.5),\quad-1\leq t\leq0,\\
				&\tilde{v}_1(t)=e^{-t}(1-t)+0.5,\ \tilde{v}_2(t)=e^{-t}(1+t)+0.5,\quad-1\leq t\leq0.
			\end{cases}
			\]
			 Note that  \(u_i(t) = u(s_i, t)\) and \(v_i(t) = v(s_i, t)\), and 
			$\left\{u_n = \left[ {\begin{array}{*{20}{c}}
					{{u_1(t_n)}}\\
					{{u_2(t_n)}}
			\end{array}} \right],v_n = \left[ {\begin{array}{*{20}{c}}
					{{v_1(t_n)}}\\
					{{v_2(t_n)}}
			\end{array}} \right]\right\}$
		and $\left\{\tilde{u}_n = \left[ {\begin{array}{*{20}{c}}
				{{\tilde{u}_1}(t_n)}\\
				{{\tilde{u}_2}(t_n)}
		\end{array}} \right],\tilde{v}_n = \left[ {\begin{array}{*{20}{c}}
				{{\tilde{v}_1}(t_n)}\\
				{{\tilde{v}_2}(t_n)}
		\end{array}} \right]\right\}$ are the numerical solutions obtained by the initial functions above, respectively.
	Fig.\ref{fig:example2} shows the disturbance errors when \(h = 0.1\) and \(h = 0.01\).
	Similarly, to more clearly illustrate the accuracy of the method, we introduce the following expressions to represent the differences between the two sets of numerical solutions
	\[E(t)=\|u_n-\tilde{u}_n\|_\infty,\quad EA(t)=\|v_n-\tilde{v}_n\|_\infty.\]The numerical results presented in Table 2 more clearly demonstrate the stability and convergence of the method. In summary, we can conclude that under the conditions specified in the text, the method is stable and asymptotically stable.
\begin{table}[htbp]
	\centering
	\caption{When numerical methods are applied to Equations \eqref{6.3}, the differences between the two solutions are \(E(t)\) and \(EA(t)\)}
	\label{tab:example2}
	\renewcommand{\arraystretch}{1.2}
	\begin{tabular}{cccccc}
		\toprule
		&h& $t=0.5$ & $t=1$ & $t=5$ & $t=10$ \\
		\midrule
		\multirow{2}{*}{$E(t)$}
	&0.1& 6.7956e-04 & 3.2452e-04 & 4.0993e-11 & 9.1818e-19 \\
		&0.01& 6.5345e-04 & 2.5050e-04 & 4.3440e-11 & 9.3851e-19 \\
		\midrule
		\multirow{2}{*}{$EA(t)$}
	&0.1	& 1.3382e-01 & 4.6812e-02 & 5.7840e-07 & 5.1553e-13 \\
		&0.01& 1.1259e-01 & 2.5012e-02 & 4.2201e-07 & 3.9258e-13 \\

		\bottomrule
	\end{tabular}
\end{table}
			\begin{figure}[htp]
				\centering
				\includegraphics{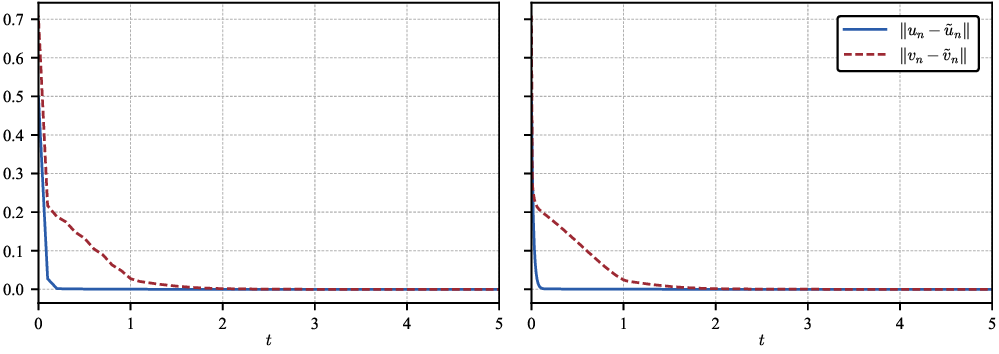}
				\caption{The perturbation errors $\|u_n-\tilde{u}_n\| \ \text{and}\  \|v_n-\tilde{v}_n\|$ with $h= 0.1$ (left) and $h= 0.01$ (right).}
				\label{fig:example2}
			\end{figure}
		\end{example}
		\section{Conclusion}
		In this paper, we  investigated the application of Runge-Kutta methods combined with compound quadrature rules for solving DIDAEs. Stability and asymptotic stability conditions for the exact solutions of DIDAEs were rigorously established. Furthermore, global and asymptotic stability conditions for CQRK methods were derived through a rigorous theoretical analysis. Numerical experiments demonstrated that the stability and asymptotic stability of DIDAEs are well preserved by the CQRK methods. They are  relatively easy to implement and integrates well with standard root-finding solvers. But CQRK methods  require solving nonlinear systems at each step, which may be computationally expensive for large-scale problems. 
		\section*{Acknowledgement}
		This work was supported by the National Natural Science Foundation of China (12271367) and Scientific Research Fund of Hunan Provincial Education Department, PR China (21A0115).
		
		\section*{Data availability}
		Data will be made available on request.

\bibliographystyle{elsarticle-num}
\bibliography{reference}
	\end{document}